\documentclass[11pt]{amsart}
\usepackage[arrow,dvips,matrix,arrow,ps,color,line,curve,frame]{xy}%%
\usepackage{amscd,amsfonts,amssymb}
\usepackage{graphicx,enumerate}
\usepackage{color}

\let\oldlabel=\label
\def\prellabel{\marginparsep=1em
    \def\label##1{\oldlabel{##1}\ifmmode\else\ifinner\else
         \marginpar{{\color{blue} \footnotesize\ \\ \tt
                    ##1}}\fi\fi}}
\let\Bbb=\mathbb

\def\Ker{\operatorname{Ker}}%%

\def\PP{{\Bbb P}}%%

\def\cA{\operatorname{\mathcal A}}

\def\cL{\operatorname{\mathcal L}}

\def\P{\mathcal P}
\def\Q{\mathcal Q}

\def\a{\mathbf a}
\def\e{\mathbf e}
\def\t{\mathbf t}
\def\v{\mathbf v}

\def\x{\mathbf x}
\def\y{\mathbf y}
\def\z{\mathbf z}
\def\m{\mathbf m}

\def\vertex{\operatorname{vert}}
\def\ehr{\operatorname{ehr}}
\def\gv{\operatorname{gv}}

\def\Proj{\operatorname{Proj}}
\def\Pic{\operatorname{Pic}}

\def\starr{\operatorname{star}}

\def\inte{\operatorname{int}}

\def\Spec{\operatorname{Spec}}%

\def\L{\operatorname{L}}%%

\def\conv{\operatorname{conv}}%

\def\gp{\operatorname{gp}}

\def\CC{{\Bbb C}}
\def\RR{{\Bbb R}}

\def\ZZ{{\Bbb Z}}
\def\NN{{\Bbb N}}

\def\kk{{\Bbb K}}

\def\Def{\textbf}

\newtheorem{lemma}{Lemma}[section]

\newtheorem{theorem}[lemma]{Theorem}
\newtheorem{proposition}[lemma]{Proposition}

\newtheorem{conjecture}[lemma]{Conjecture}

\theoremstyle{definition}
\newtheorem{definition}[lemma]{Definition}

\newtheorem{example}[lemma]{Example}

\begin{document}

\title{Very ample and Koszul segmental fibrations} % [Koszul segmental fibrations]

\author{Matthias Beck}
\author{Jessica Delgado}
\author{Joseph Gubeladze}
\author{Mateusz Micha$\l$ek}

\address{Department of Mathematics\\
         San Francisco State University\\
         1600 Holloway Ave.\\
         San Francisco, CA 94132, USA}
\email{[mattbeck,soso]@sfsu.edu}
\address{Kahikaluamea Department \\
4303 Diamond Head Rd. \\
Honolulu, HI 96816, USA}
\email{Delgado3@hawaii.edu}
\address{Polish Academy of Sciences\\
         ul. \'Sniadeckich 8\\
         00-956 Warsaw\\
         Poland}
\email{wajcha2@poczta.onet.pl}

\thanks{This research was partially supported by the U.\ S.\ National Science Foundation through the grants DMS-1162638 (Beck), DGE-0841164 (Delgado),
DMS-1000641 \& DMS-1301487 (Gubeladze), and the Polish National Science Centre grant no 2012/05/D/ST1/01063 (Micha{\l}ek)}

\date{5 December 2014}

\maketitle

\begin{abstract}
In the hierarchy of structural sophistication for lattice polytopes, normal polytopes mark a point of origin; very ample and Koszul polytopes occupy bottom and top spots in this hierarchy, respectively.

In this paper we explore a simple construction for lattice polytopes with a twofold aim. On the one hand, we derive an explicit series of very ample 3-dimensional
polytopes with arbitrarily large deviation from the normality property, measured via the highest discrepancy degree between the corresponding Hilbert functions and
Hilbert polynomials. On the other hand, we describe a large class of Koszul polytopes of arbitrary dimensions, containing many smooth polytopes and extending the previously known class of Nakajima polytopes.
\end{abstract}

%*************************************************

\section{Introduction}\label{Intro}

Normal polytopes show up in various contexts, including algebraic geometry (via toric varieties), integer programming (integral Carath\'eodory property), combinatorial commutative
algebra (quadratic Gr\"ob\-ner bases of binomial ideals), and geometric combinatorics (Ehr\-hart theory). Section \ref{ClassesofPolytopes} below introduces all relevant classes of polytopes and provides background material which also serves as motivation.

The weakest of the general properties, distinguishing a polytope from random lattice polytopes, is \emph{very ampleness}. In geometric terms, very ample polytopes
correspond to normal projective but not necessarily projectively normal embeddings of toric varieties. The finest of the algebraic properties a lattice polytope can have is
the \emph{Koszul} property, which means that the polytope in question is homologically wonderful; the name draws its origin from Hilbert's Syzygy Theorem. It says that (the algebra of) a unit lattice simplex, i.e., the polynomial algebra with the vertices of the simplex as variables, is Koszul with the standard Koszul complex on the vertices as a certificate.

Koszul algebras were discovered by topologists in the early 1970s and, with the advent of powerful theoretical, computational, and practical tools, the topic of
Koszul polytopal algebras became a popular topic in algebraic combinatorics starting from the early 1990s.

By contrast, the very ample polytopes, formally treatable
within the framework of elementary additive number theory, came under spotlight more recently. Partly this can be explained by the traditional inclusion of the
normality property in the definition of toric varieties, so a property weaker than normality seemed unnatural. The very question of existence of very ample non-normal polytopes became interesting only in the late 1990s.

There is a whole panorama of interesting classes of lattice polytopes sandwiched between the very ample and Koszul ones, but we are not delving into them. In this
paper, by exploring a simple polytopal construction called \emph{lattice segmental fibrations} (Definition \ref{fibration} below), we detect very ample lattice polytopes, arbitrarily far away from the normality property (Theorem \ref{highgaps}), and construct a new large class of Koszul polytopes (Theorem \ref{koszulclass}), containing many examples of smooth polytopes. A subclass of Koszul polytopes was described in \cite{DHZ} and in Section \ref{Koszul} we explain that the argument there works for our general class as well.

\bigskip\noindent\emph{Acknowledgement.} We thank Winfried Bruns and Serkan Ho\c{s}ten for helpful comments and providing us with an invaluable set
of examples of very ample polytopes. We also thank Milena Hering for pointing out the overlap of our work with \cite{DHZ} and two anonymous referees for helpful
comments. The last author also thanks Mathematisches Forschungsinstitut Oberwolfach for the great working atmosphere and hosting.

\section{Normal, very ample, and Koszul polytopes}\label{ClassesofPolytopes}

In this section we introduce three polytopal classes and give a characterization of the very ampleness condition (Proposition \ref{veryample}), which we were not able to find in the literature in the given generality.

\subsection{Normal polytopes}\label{NormalPolytopes}

A \Def{(convex) polytope} $\P$ is the convex hull of a finite subset of $\RR^d$. The inclusion-minimal such set
$\vertex(\P)$ consists of the \Def{vertices} of $\P$.
If $\vertex(\P)$ is affinely independent, $\P$ is a \Def{simplex}.
A polytope $\P\subset\RR^d$ is \Def{lattice} if $\vertex(\P$)$\subset\ZZ^d$.
For a lattice polytope $\P\subset\RR^d$ we denote by $\L(\P)$ the subgroup of $\ZZ^d$ that is affinely generated by the lattice points in $\P$, i.e.,
$$
\L(\P)=\sum_{\x, \, \y\in \P\cap\ZZ^d}\ZZ(\x-\y)\subset\ZZ^d.
$$
A lattice polytope $\P\subset\mathbb{R}^d$ is called
\begin{enumerate}[(a)]
\item \Def{integrally closed} if for every natural number $c$ and every point $\z\in c \, \P\cap\ZZ^d$ there exist
$\x_1,\x_2,\ldots,\x_c\in \P\cap\ZZ^d$ such that $\x_1+\x_2+\dots+\x_c=\z$;
\item \Def{normal} if for some (equivalently, every) point $\t\in \P\cap \mathbb{Z}^d$ the following condition is satisfied: for every natural number
$c$ and every point $\z\in c \, \P\cap(c\t+\L(\P))$ there exist $\x_1,\x_2,\ldots,\x_c\in \P\cap\ZZ^d$ such that $\x_1+\x_2+\dots+\x_c=\z$.
\end{enumerate}
\noindent(For $\t$ in (b), we have $\P\cap(\t+\L(\P))=\P\cap\ZZ^d$.)

One easily observes that a lattice polytope $\P\subset\RR^d$ is integrally closed if and only if it
is normal and $\L(\P)$ is a direct summand of $\ZZ^d$. In particular, a normal polytope $\P$ becomes full-dimensional
and integrally closed if one changes the ambient space $\RR^d$ to $\RR\L(\P)$ and the lattice of reference to $\L(\P)$.
This explains why the difference between \emph{normal} and \emph{integrally closed} is often blurred in the literature.

An \Def{empty} simplex (i.e., a lattice simplex that contains no lattice points besides its vertices) of large volume
provides an example of a normal but not integrally closed polytope. The classification of empty simplices is an
active area of research.
One class of empty simplices is important for this paper: a lattice $n$-simplex $\conv(\x_0,\x_1,\ldots,\x_n)\subset\RR^d$ is \Def{unimodular} if $\{\x_1-\x_0,\ldots,\x_n-\x_0\}$ is a part of a basis of $\ZZ^d$. Unimodular
simplices are integrally closed, and if a lattice polytope is a union of unimodular simplices, i.e., admits a \Def{unimodular cover}, then it is integrally closed. Not
all lattice 4-polytopes with a unimodular cover admit triangulations into unimodular simplices---an example \cite[Proposition 1.2.4(c)]{BrGuTr} has been shown by
effective methods to have a unimodular cover---and not all integrally closed 5-polytopes are covered by unimodular simplices~\cite{uncovered}.

\subsection{Point configurations and Proj}\label{Proj}

Our next two classes of polytopes are best explained by providing an algebraic context.
We denote the \Def{conical hull} of $X\subset\RR^d$ by $\RR_{\ge0}X$. An \Def{affine monoid} is a finitely
generated additive submonoid of $\ZZ^d$ for some $d\in\NN$. For a finite subset $X=\{\x_1,\ldots,\x_n\}\subset\ZZ^d$ the affine monoid generated by $X$ will be denoted by
\[
  \ZZ_{\ge0}X=\ZZ_{\ge0}\, \x_1+\cdots+\ZZ_{\ge0}\, \x_n \, .
\]
The group of differences of an affine monoid $M\subset\ZZ^d$ (i.e., the subgroup of $\ZZ^d$ generated by $M$) is denoted by $\gp(M)$, and the \Def{normalization} of
$M$ is the following affine monoid
$$
\overline M:=\gp(M)\cap\RR_{\ge0}M=\{\x\in\gp(M)\ |\ n \, \x\in M\ \text{for some}\ n\in\NN\} \, .
$$
For a field $\kk$ and an affine monoid $M\subset\ZZ^d$, the normalization of the monoid ring $\kk[M]$  (i.e., its integral closure in the field of fractions) equals
$\kk[\overline M]$ (see, e.g., \cite[Section 4.E]{Kripo}). The monoid algebra $\kk[M]$ can be thought of as the monomial subalgebra of $\kk[X_1^{\pm1},\ldots,X_d^{\pm1}]$ spanned by the Laurent monomials with exponent vectors in~$M$.

We will refer to a finite subset $\cA\subset\ZZ^d$ as a \Def{point configuration}.  For a point configuration $\cA$ and a field $\kk$, we let $\kk[\cA]$ denote the monoid $\kk$-algebra of the affine monoid
\[
  M_{\cA}:=\sum_{\x \in \cA}\ZZ_{\ge0}(\x,1)\subset\ZZ^{d+1}.
\]
It is natural to call the last coordinate of a point $\y \in \overline{M_{\cA}}$ the \Def{height} of~$\y$.

A \Def{grading} on an affine monoid $M$ is a partition $M=\bigcup_{i \in \ZZ_{\ge0}}M_i$, where $M_0=\{0\}$ and $M_i+M_j\subset
M_{i+j}$. For a field $\kk$ and a graded affine monoid $M$, the monoid algebra $\kk[M]$ is graded in the natural way, where the $0$th component equals $\kk$.

The monoids of type $M_{\cA}$, where $\cA$ is a point configuration, are naturally graded with respect to the last
coordinate, and they are generated in degree 1. In particular, for a field $\kk$ and a point configuration $\cA$, the graded algebra $\kk[\cA]$ is \Def{homogeneous}:
\[
\kk[\cA]=\kk\oplus A_1\oplus A_2\oplus\cdots,\qquad \kk[\cA]=\kk[A_1] \, , \qquad A_1=\kk(\cA,1) \, .
\]

For a lattice polytope $\P\subset\RR^d$, the corresponding \Def{polytopal monoid} is defined by $M_\P:=M_{\cA}$ where $\cA=\P\cap\ZZ^d$.
We denote $\L(\cA):=\sum_{\x,\y\in\cA}\ZZ(\x-\y)\subset\ZZ^d$. Thus, using the previous notation, for a lattice polytope $\P\subset\ZZ^d$, we have $\L(\P)=\L(\P\cap\ZZ^d)$.

For an algebraically closed field $\kk$ and a point configuration $\cA\subset\ZZ^d$, we have the projective variety
\[
  X_{\cA}:=\Proj(\kk[\cA])\subset\PP^{N-1}_\kk \, ,
\]
where $N=\#\cA$, and the resulting very ample line bundle $\cL_{\cA}\in\Pic(X_{\cA})$ (see, e.g., \cite[Ch.~II,~\S7]{Algeo}).

The following proposition is a semi-folklore result.

\begin{proposition}\label{veryample}
Let $\kk$ be an algebraically closed field, $\cA\subset\ZZ^d$ a point configuration, and $\P:=\conv(\cA)$. Then the following are equivalent:
\begin{enumerate}[{\rm (a)}]
\item $X_{\cA}$ is normal;
\item $\bigoplus_{i\ge0}H^0(X_{\cA},\cL_{\cA}^{\otimes i})=\kk\big[\RR_{\ge0}(\P,1)\cap\gp(M_{\cA})\big]$;
\item $\RR_{\ge0}(\P-\v)\cap\L(\cA)=\ZZ_{\ge0}(\cA-\v)$ for every $\v\in\vertex(\P)$;
\item $\#(\RR_{\ge0}(\P,1)\cap\gp(M_{\cA}))\setminus M_{\cA}<\infty$;
\item $X_{\cA}$ is a projective toric variety.
\end{enumerate}
\end{proposition}

\begin{proof} (a)$\Longleftrightarrow$(b) holds because the left-hand side of (b) is the normalization of
$\kk[\cA]$---a general fact for a normal projective variety and a very ample line bundle on it \cite[Ch. II, Exercise 5.14]{Algeo}, while the right-hand side is $\kk[\overline{M_{\cA}}]$.

\medskip \noindent (a)$\Longleftrightarrow$(c) follows from the open affine cover
$$
  X_{\cA}=\bigcup_{\v \in \vertex(\P)} \Spec(\kk[\ZZ_{\ge0}(\cA-\v)])
$$
\cite[Proposition 2.1.9]{TORICvarieties} and the fact that, for every vertex $\v\in \P$, the normalization of
$\kk[\ZZ_{\ge0}(\cA-\v)]$ equals $\kk[\RR_{\ge0}(\P-\v)\cap\L(A)]$ \cite[Section 4.E]{Kripo}.
The aforementioned affine cover, in turn, follows from the standard
dehomogenization with respect to the degree-1 generating set $(\cA,1)\subset\kk[\cA]$ and the observation that
the affine charts $\Spec(\kk[\ZZ_{\ge0}(\cA-\x)])$ with $\x\in\cA\setminus\vertex(\P)$ are redundant: if $\x\in\inte(F)$ for a positive-dimensional face $F\subset \P$, then we have
\begin{align*}
\kk[\ZZ_{\ge0}(\cA-\x)]=\kk[\ZZ_{\ge0}(\cA-\z)-\ZZ_{\ge0}((\cA\cap F)-\z)] \, ,
\end{align*}
where $\z$ is any vertex of $F$ and the right-hand side is the localization with respect to the monomial multiplicative subset
$\ZZ_{\ge0}((\cA\cap F)-\z)\subset \kk[\ZZ_{\ge0}(\cA-\z)]$.

\medskip \noindent (a)$\Longleftrightarrow(d)$ This is Theorem 13.11 in \cite{STURMpol}.

\medskip \noindent (a)$\Longleftrightarrow$(e) This is, essentially, a matter of convention: some sources (e.g., \cite{Kripo,Toric}) include the normality in the
definition of a toric variety, whereas the recent comprehensive reference in the field \cite{TORICvarieties} relaxes this assumption.
\end{proof}

\subsection{Very ample polytopes}\label{VeryAmplePolytopes}

In view of Proposition \ref{veryample}, it is natural to call a point configuration $\cA\subset\ZZ^d$ \Def{very
ample} if it satisfies the equivalent conditions in the proposition. If $\cA$ is very ample, the elements of $\overline{M_{\cA}}\setminus M_{\cA}$ will be called \Def{gaps}, and the maximal possible degree of a gap will be denoted by $\gamma(\cA)$, i.e.,
$$
\gamma(\cA)=
\begin{cases}
&0,\ \text{if}\ (M_{\cA})_i=(\overline{M_{\cA}})_i\ \text{for all}\ i\in\ZZ_{\ge0},\\
&\max\left(i\ |\ (M_{\cA})_i\subsetneq(\overline{M_{\cA}})_i\right)\in\NN,\ \text{otherwise}.
\end{cases}
$$

For $\cA$ very ample, $\gamma(\cA)$ is a higher-dimensional analog of the \emph{Frobenius number of a numerical
monoid} (see, e.g., \cite{DIOPHANTINE}): there are no gaps in the degrees $>\gamma(\cA)$. The analogy is limited though---the monoid $M_{\cA}$ is generated in degree
1, whereas the classical Frobenius number of a numerical monoid $M\subset\ZZ_{\ge0}$ is not defined exactly in the situation when $M$ is generated in degree 1, i.e., when $M=\ZZ_{\ge0}$.

Since $\gamma(\cA)$ depends only on the monoid $M_{\cA}$ and not on how $\cA$ sits in $\ZZ^d$, without loss of
generality we can assume $\L(\cA)=\ZZ^d$.
This is achieved by changing the original ambient lattice $\ZZ^d$ to $\L(\cA)$. The upshot of this assumption is that the Hilbert function of $\kk[\overline{M_{\cA}}]$ is now the \Def{Ehrhart polynomial} of $\P=\conv(\cA)$, i.e.,
$$
\dim_\kk(\kk[\overline{M_{\cA}}])_j)= \# \left( j\P \cap \ZZ^d \right) =: \ehr_\P(j),\qquad j\in\NN.
$$

Next we observe that $\gamma(\cA)$ can be made arbitrarily large by varying $\cA$ without changing $X_{\cA}$. In
fact, for the `rarified' very ample configurations
$$
\cA_c=\bigcup_{\v \in \vertex(\P)}((c-1)\v+\cA),\quad c\in\NN,
$$
we have
\begin{align*}
&\conv(\cA_c)=c\cdot\conv(\cA)\, , \\
&X_{\cA_c}\cong X_{\cA},\quad c\in\NN \, , \\
&\gamma(\cA_c)\to\infty\ \ \text{as}\ \ c\to \infty \, .
\end{align*}
These observations explain why one needs to restrict to very ample polytopes in the quest for upper bounds for $\gamma(\cA)$:
a \Def{very ample} polytope is a lattice polytope $\P\subset\RR^d$ such that the point configuration
$\P\cap\ZZ^d\subset\ZZ^d$ is very ample and $\L(\P)$ is a direct summand of $\ZZ^d$.

For a very ample polytope $\P\subset\RR^d$ we denote $\gamma(\P)=\gamma(\P\cap\ZZ^d)$. The number $\gamma(\P)$ measures how far the embedding
$X_{\P\cap\ZZ^d}\hookrightarrow\PP^{N-1}_\kk$, where again $N=\#(\P\cap\ZZ^d)$, is from being projectively normal. Alternatively, the
number $\gamma(\P)$ can be defined as the maximal degree beyond which the Hilbert function of $\kk[\P\cap\ZZ^d]$ equals
its Hilbert polynomial.

For a lattice polytope $\P\subset\RR^d$ the smallest generating set of the normalization
$\overline{M_{\P\cap\ZZ^d}}$ is called the \Def{Hilbert basis}. It is
concentrated in degrees $<\dim(\P)$ \cite[Theorem 2.52]{Kripo}. For $\P$ very ample, the elements of the Hilbert basis of $\overline{M_{\P\cap\ZZ^d}}$
represent gaps in $\overline{M_{\P\cap\ZZ^d}}$. One might expect that, similarly, there is a dimensionally uniform upper bound for the degrees of \emph{all} gaps.
However, we show in Section \ref{Class} that this is false already in dimension three, even for very ample lattice polytopes with eight lattice points (see
Theorem \ref{highgaps} below).

Our extremal examples suggest that it might be of interest to study the \Def{gap vector}
$\gv(\P)$ of a very ample polytope $\P$, with entries
\[
  \gv_k(\P) := \# \text{ gaps in $M_\P$ at height $k$,}
\]
stopping at the largest height $\gamma(\P)$ that contains gaps in $M_\P$.
As an indication that this might be an interesting concept, we offer a conjecture on unimodality of gap vectors (see
Conjecture \ref{gapvectorconj} below) and verify it for the gap vectors for a
family of polytopes that play a central role in Section~\ref{Class}.

In the proof of Proposition \ref{veryample} we described the monomial affine charts of $\Proj(\kk[\cA])$. This description implies
that, for a lattice polytope $\P\subset\RR^d$ with $\L(\P)$ a direct summand of $\ZZ^d$, the variety $\Proj(\kk[M_\P])$
is smooth if and only if the primitive edge vectors at every vertex of $\P$ define a part of a basis of $\ZZ^d$ \cite[Exercise 4.25]{Kripo}.
Correspondingly, one calls such polytopes \Def{smooth}. Clearly, smooth polytopes are very ample. Much effort went into the study whether $\gamma(\P)=0$ for a
smooth polytope $\P$, i.e., whether smooth polytopes are integrally closed. This is the well-known \Def{Oda question}, still wide open, even in dimension three \cite{mfo}.

\subsection{Koszul polytopes}\label{KoszulPolytopes}

Let $\kk$ be a field. A finitely generated graded $\kk$-algebra $\Lambda=\kk\oplus\Lambda_1\oplus\cdots$ is \Def{Koszul} if the minimal free graded resolution of $\kk$ over $\Lambda$ is linear:
\[
\cdots\mathrel{\mathop{\longrightarrow}\Lambda^{\beta_2}}
\mathrel{\mathop{\longrightarrow}^{\partial_2}}\Lambda^{\beta_1}\mathrel{\mathop{\longrightarrow}^{\partial_1}}\Lambda\mathrel{\mathop{\longrightarrow}^{\partial_0}}\kk
\mathrel{\mathop{\longrightarrow}}0,\quad\deg(\partial_i)=1,\quad i>0.
\]
The condition $\deg(\partial_1)=1$ is equivalent to $\Lambda$ being homogeneous and the condition $\deg(\partial_1)=\deg(\partial_2)=1$
is equivalent to $\Lambda$ being \Def{quadratically defined}, i.e.,
\begin{align*}
\Lambda=\kk[X_1,\ldots,X_N]/(F_1,\ldots,F_n),\quad N=\dim_\kk\Lambda_1,
\end{align*}
for some homogeneous quadratic polynomials $F_1,\ldots,F_n$.

When $\Lambda$ is a quadratically defined graded monoid algebra $\kk[M]$, then the polynomials
$F_1,\ldots,F_n$ can be chosen to be of the form $m-m'$ for some degree-2 monomials $m,m'\in\kk[X_1,\ldots,X_N]$; see \cite[Sections 4.A,B,C]{Kripo} for generalities on monoid algebras.

A well-known sufficient (but in general not necessary) criterion for the Koszul property, already detected in
Priddy's pioneering work on Koszul algebras \cite{Priddy}, is the existence of a quadratic Gr\"obner basis; for a proof using Gr\"obner-bases terminology see, e.g., \cite{Vetter}. Namely, a $\kk$-algebra
$\Lambda=\kk[X_1,\ldots,X_N]/I$, where $I$ is a homogeneous ideal, is Koszul if $I$ admits a
quadratic Gr\"obner basis with respect to some term order on $\kk[X_1,\ldots,X_N]$.

Oda's question, mentioned above, corresponds to the degree-$1$ part of \Def{B{\o}gvad's conjecture}, which claims that for every
smooth lattice polytope $\P\subset\RR^d$, the algebra $\kk[\RR_{\ge0}(\P,1)\cap\ZZ^{d+1}]$ is Koszul.

We call a lattice polytope $\P$ \Def{quadratically defined} or \Def{Koszul} if the graded monoid algebra
$\kk[\P\cap\ZZ^d]$ is quadratically defined or Koszul, respectively, for every field $\kk$. The property of being quadratically defined is independent of $\kk$,
but whether $\kk[\P\cap\ZZ^d]$ being Koszul depends on $\kk$ is an open question \cite[Question 8.5.6]{Peeva}. In
particular, if Oda's question has a positive answer, then B{\o}gvad's conjecture is equivalent to the claim that smooth polytopes are Koszul.

Examples of Koszul polytopes (or point configurations) include:
\begin{enumerate}[{\rm$\bullet$}]
\item The dilated lattice polytopes $c \, \P$ for $c\ge\dim \P$ \cite[Theorem 1.3.3]{BrGuTr}; for
sharper lower bounds for $c$, depending on $\P$, see \cite[Section 4]{Hering}.
\item Lattice polytopes cut out by root systems of classical type and their Cayley sums \cite{PAYNEroot}; type $A$ was considered before in \cite[Theorem 2.3.10]{BrGuTr}, using different methods; see Example \ref{Example2} below.
\item The non-polytopal point configuration $\cA=\conv(0,3\e_1,3\e_2,3\e_3)\setminus\{(1,1,1)\}$ \cite{Caviglia}.
\end{enumerate}

Recently, Oda's question and B{\o}gvad's conjecture (and extensions to more general point configurations)  have been
attracting considerable interest in the community of algebraic combinatorics \cite{aim,mfo}. For an effective approach to a potential counterexample
to B{\o}gvad's conjecture, see~\cite{BRUNSquest}. The surveys
\cite{CONCA,Froberg,Peeva} include much relevant general background material.

In Section \ref{Koszul}, we derive a new large
class of Koszul polytopes in arbitrary dimensions. In particular, when our examples of very ample 3-polytopes in
Section \ref{Class} below happen to be smooth, then they are normal and Koszul as well.

Note that if $\P$ is integrally closed (resp.\ normal, very ample, Koszul) then so are the faces of $\P$; for the Koszul property one uses \cite[Proposition 1.4]{Ohsugi}.
One can also show that if $\P$ and $\Q$ are integrally closed (resp.\ normal, very ample, Koszul) then so is $\P\times\Q$; the ring $\kk[M_{\P\times\Q}]$ is the \Def{Segre product} of $\kk[\P]$ and $\kk[Q]$ and the Koszul property transfers from factor algebras to their Segre product \cite[Theorem 2(ii)]{Froberg}. In particular, since there are 3-dimensional non-normal very ample polytopes (see Section \ref{Class}), the direct product with the unit segment $[0,1]$ yields the existence of non-normal very ample polytopes in all dimensions $d\ge3$. Classically, all lattice $d$-polytopes
with $d\le2$ are integrally closed (see, e.g., \cite[Proposition 1.2.4]{BrGuTr}). Moreover, by \cite[Corollary 3.2.5]{BrGuTr} a lattice polygon is Koszul if and only if either it is a unimodular triangle or it has at least $4$ lattice points in the boundary.

\section{Very ample 3-polytopes with gaps of arbitrarily large degrees}\label{Class}

The polytopal construction announced in the introduction and central to this paper is as follows:

\begin{definition}\label{fibration}
An affine map $f:\P\to \Q$ between lattice polytopes $\P\subset\ZZ^{d_1}$ and $\Q\subset\ZZ^{d_2}$ is a \Def{lattice segmental fibration} if it satisfies the following conditions:
\begin{enumerate}[{\rm (i)}]
\item $f^{-1}(\x)$ is a lattice segment, i.e., a one-dimensional lattice polytope or a lattice point, for every $\x\in \Q\cap\ZZ^{d_2}$,
\item $\dim(f^{-1}(\x))=1$ for at least one $\x\in \Q\cap\ZZ^{d_2}$,
\item $\P\cap\ZZ^{d_1}\subset\bigcup_{\Q\cap\ZZ^{d_2}}f^{-1}(\x)$.
\end{enumerate}
\end{definition}
%
% \vspace{-.2in}
%\newpage
\begin{figure}[h!]
\caption{A lattice segmental fibration.}
\vspace{.2in}
\includegraphics[trim = 0mm 1.25in 0mm 1.5in, clip, scale=.2]{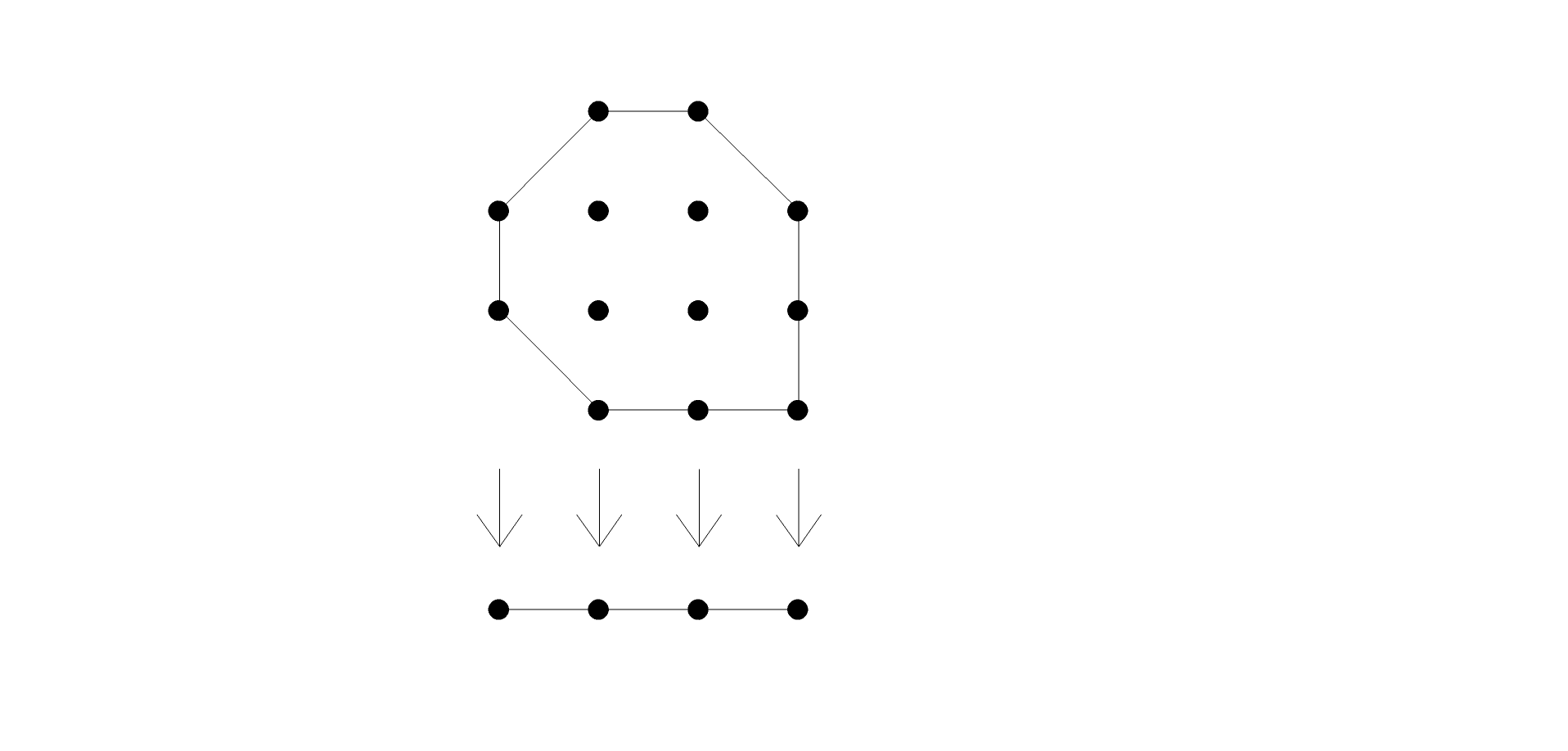}
\end{figure}

It follows from this definition that a lattice segmental fibration $f:\P\to\Q$ is a surjective map and we have the isomorphism of groups $\L(\P)\cong\L(\Q)\oplus\ZZ$.

In this section, using certain small lattice segmental fibrations, we show that there is no uniform upper bound for $\gamma(\P)$ even for 3-dimensional very ample polytopes with a few lattice points.

The following class of 3-polytopes was introduced in \cite[Exercise 2.24]{Kripo}. The first explicit representatives of the class showed up already in Bruns' report
\cite[p.~2290]{mfo}. Let
$I_k=[a_k,b_k]\subset\RR$ be lattice segments for $1\le k\le 4$, none of them degenerated to a point. Let
\[
  \P(I_1,I_2,I_3,I_4):=\conv\big((0,0,I_1),(1,0,I_2),(0,1,I_3),(1,1,I_4)\big)\subset\RR^3 .
\]
Thus the map
\begin{align*}
\P(I_1,I_2,I_3,I_4)\to\conv\big((0,0),(1,0),(0,1),(1,1)\big),\quad (x,y,z)\mapsto(x,y),
\end{align*}
is a lattice segmental fibration.
\begin{lemma}\label{Kripoexercise}
\begin{enumerate}[{\rm (a)}]
\item $\P(I_1,I_2,I_3,I_4)$ is very ample \cite[Exercise 2.24]{Kripo}.
\item  $\P(I_1,I_2,I_3,I_4)$ is smooth if and only if $a_1+a_4=a_2+a_3$ and $b_1+b_4=b_2+b_3$.
\end{enumerate}
\end{lemma}
\begin{proof}[Proof of Lemma \ref{Kripoexercise}] (a)
Acting by translations and lattice automorphisms we can assume $I_1=[0,b_1]$ and we only need to check that
\begin{equation}\label{veryampleequality}
C\cap\ZZ^3=\ZZ_{\ge0}(1,0,a_2)+\ZZ_{\ge0}(0,1,a_3)+\ZZ_{\ge0}(1,1,a_4)+\ZZ_{\ge0}\e_3 \, ,
\end{equation}
where $C=\RR_{\ge0}\P(I_1,I_2,I_3,I_4)$ and $\e_3=(0,0,1)$.
There are two possibilities: either
\begin{align*}
&C=\RR_{\ge0}\e_3+\RR_{\ge0}(1,0,a_2)+\RR_{\ge0}(0,1,a_3),\ \text{or}\\
&C=\big(\RR_{\ge0}\e_3+\RR_{\ge0}(1,0,a_2)+\RR_{\ge0}(1,1,a_4)\big)\ \cup\ \big(\RR_{\ge0}\e_3+\RR_{\ge0}(0,1,a_3)+\RR_{\ge0}(1,1,a_4)\big).
\end{align*}
In the first case, (\ref{veryampleequality}) holds because $\{\e_3,(1,0,a_2),(0,1,a_3)\}$ is a basis of $\ZZ^3$:
\[
\det\left[ \begin{array}{ccc}
1&0&a_2\\
0&1&a_3\\
0&0&1
  \end{array} \right]=1.
\]
In the second case, (\ref{veryampleequality}) holds because the two cones on the right-hand side are spanned by bases of~$\ZZ^3$:
\[
\det\left[ \begin{array}{ccc}
1&0&a_2\\
1&1&a_4\\
0&0&1
  \end{array} \right]=1,
\qquad
\det\left[ \begin{array}{ccc}
0&1&a_3\\
1&1&a_4\\
0&0&1
  \end{array} \right]=-1,
\]
and therefore
\[
C\cup\ZZ^d=\big(\ZZ_{\ge0}(1,0,a_2)+\ZZ_{\ge0}(1,1,a_4)+\ZZ_{\ge0}\e_3\big)\, \cup\, \big(\ZZ_{\ge0}(0,1,a_3)+\ZZ_{\ge0}(1,1,a_4)+\ZZ_{\ge0}\e_3\big).
\]

\medskip \noindent (b) The argument in part (a) shows that if a vertex of $\P(I_1,I_2,I_3,I_4)$ is simple then the primitive edge vectors at this vertex form a basis of $\ZZ^3$. So the polytope $\P(I_1,I_2,I_3,I_4)$ is smooth if and only if it is simple. On the other hand, $\P(I_1,I_2,I_3,I_4)$ being simple means the `bottom' vertices
$(0,0,a_1)$, $(1,0,a_2)$, $(0,1,a_3)$, $(1,1,a_4)$ align in a plane (i.e., they span a facet) and so do the `top' vertices $(0,0,b_1)$, $(1,0,b_2)$, $(0,1,b_3)$, $(1,1,b_4)$. This is equivalent to the desired equalities.
\end{proof}
An extension of (\ref{veryampleequality}) in the proof of Lemma \ref{Kripoexercise}(a) for $d$-dimensional
rational cones $C\subset\RR^d$ with $d+1$ extremal vectors was given in \cite[Proposition 8.1]{BRUNSquest}. The class of smooth polytopes of type $\P(I_1,I_2,I_3,I_4)$ will be extended in Example \ref{Example} below.

We now specialize to the family of polytopes
\begin{align*}
  \P_m &:= \P \left( [0,1], \, [0,1], \, [0,1], \, [m,m+1] \right) .
\end{align*}

\def\JPicScale{.7}
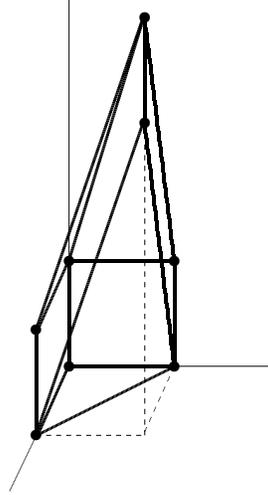
\begin{figure}[h!]
\caption{The polytope $\mathcal{P}_m$.}
\ifx\JPicScale\undefined\def\JPicScale{1}\fi
\unitlength \JPicScale mm
\begin{picture}(51.25,94.38)(0,0)
\linethickness{0.05mm}
\put(12.5,24.38){\line(0,1){70}}
\linethickness{0.1mm}
\multiput(1.25,0.62)(0.12,0.25){94}{\line(0,1){0.25}}
\linethickness{0.05mm}
\put(12.5,24.38){\line(1,0){38.75}}
\linethickness{0.3mm}
\put(6.25,31.25){\circle*{1.88}}

\linethickness{0.3mm}
\put(6.25,11.25){\circle*{1.88}}

\linethickness{0.3mm}
\put(32.5,44.38){\circle*{1.88}}

\linethickness{0.3mm}
\put(32.5,24.38){\circle*{1.88}}

\linethickness{0.3mm}
\put(12.5,44.38){\circle*{1.88}}

\linethickness{0.3mm}
\put(12.5,24.38){\circle*{1.88}}

\linethickness{0.4mm}
\put(12.5,24.38){\line(0,1){20}}
\linethickness{0.4mm}
\put(12.5,24.38){\line(1,0){20}}
\linethickness{0.4mm}
\put(32.5,24.38){\line(0,1){20}}
\linethickness{0.4mm}
\put(12.5,44.38){\line(1,0){20}}
\linethickness{0.4mm}
\put(6.25,11.25){\line(0,1){20}}
\linethickness{0.4mm}
\multiput(6.88,11.88)(0.12,0.27){47}{\line(0,1){0.27}}
\linethickness{0.4mm}
\multiput(6.88,31.88)(0.12,0.27){47}{\line(0,1){0.27}}
\linethickness{0.1mm}
\multiput(26.88,11.88)(0.87,1.92){7}{\multiput(0,0)(0.11,0.24){4}{\line(0,1){0.24}}}
\linethickness{0.05mm}
\multiput(26.88,11.25)(0,2.01){30}{\line(0,1){1.01}}
\linethickness{0.3mm}
\put(26.88,90.62){\circle*{1.88}}

\linethickness{0.3mm}
\put(26.88,70.62){\circle*{1.88}}

\linethickness{0.4mm}
\put(26.88,70.62){\line(0,1){20}}
\linethickness{0.4mm}
\multiput(6.25,11.25)(0.12,0.35){172}{\line(0,1){0.35}}
\linethickness{0.4mm}
\multiput(6.25,11.25)(0.24,0.12){109}{\line(1,0){0.24}}
\linethickness{0.4mm}
\multiput(26.88,70.62)(0.12,-0.98){47}{\line(0,-1){0.98}}
\linethickness{0.4mm}
\multiput(26.88,90.62)(0.12,-0.98){47}{\line(0,-1){0.98}}
\linethickness{0.4mm}
\multiput(12.5,44.38)(0.12,0.39){120}{\line(0,1){0.39}}
\linethickness{0.4mm}
\multiput(6.25,31.25)(0.12,0.35){172}{\line(0,1){0.35}}
\linethickness{0.05mm}
\multiput(6.25,11.25)(1.9,0){11}{\line(1,0){0.95}}
\end{picture}
\end{figure}

\noindent The underlying point configuration, written as a matrix, is
\[
  \P_m\cap\ZZ^3 = \left[ \begin{array}{cccccccc}
%  1 & 1 & 1 & 1 & 1 & 1 & 1 & 1 \\
  0 & 1 & 0 & 0 & 1 & 0 & 1 & 1 \\
  0 & 0 & 1 & 0 & 0 & 1 & 1 & 1 \\
  0 & 0 & 0 & 1 & 1 & 1 & m & m+1
  \end{array} \right] .
\]
The same  family is featured in \cite[Example 15]{8authors} in connection with the (lack of the) Gorenstein property. The monomial realizations of the corresponding monoid rings are
$$
\kk\left[\P_m\cap\ZZ^3\right]\cong\kk[Z,XZ,YZ,WZ,XWZ,YWZ,XYW^mZ,XYW^{m+1}Z] \, .
$$

Recall that the gap vector $\gv(\P)$ of a very ample polytope $\P$ has entries
\[
  \gv_k(\P) := \# \text{ gaps in $M_\P$ at height $k$,}
\]
stopping at the largest height $\gamma(\P)$ that contains gaps in $M_\P$.

We can give an explicit formula for the gap vectors $\gv(\P_m)$ for all $m$.

\begin{theorem}\label{gapvectorofPm}
Let $m\geq 3$.
The gap vector of $\P_m$ has the entries
\[
  \gv_k(\P_m) = {k+1\choose3}(m-k-1).
\]
In particular, $\gamma(\P_m) = m-2$ and
\[
\gv_1(\P_m)\le\cdots\le\gv_j(\P_m)\ge\gv_{j+1}(\P_m)\ge\cdots\ge\gv_{m-2}(\P_m)
\]
for $j=\lceil{\frac{3m-5}4}\rceil$.
\end{theorem}

Since for $\P$ very ample, $\P\times[0,1]$ is also very ample, the polytopes $\P_m$ imply the existence of non-normal very ample polytopes in all dimensions $\ge3$ with an arbitrarily large number of degree-2 gaps. Similar topics are discussed in \cite{Akihiro}.
In addition, \cite{katthan} gives examples of very ample 3-polytopes $\P$ with arbitrarily deep
gaps, measured via lattice distance relative to the facets of the cone $\RR_{\ge 0}M_{\P}$. However, the gaps in all examples constructed in \cite{Akihiro,katthan} are concentrated in degree~2.

For the proof of Theorem \ref{gapvectorofPm} we will need several auxiliary results.

\begin{lemma}\label{highgaps}
The Ehrhart polynomial of the polytope $P_m$ equals
$$\ehr_{ \P_m } (j) = \left( \tfrac m 6 + 1 \right) j^3 + 3 \, j^2 + \left( 3 - \tfrac m 6 \right)  j + 1 \, .$$
\end{lemma}

\begin{proof}
The volume of $\P_m$ is easily seen to be $\frac m 6 + 1$;
furthermore, $\P_m$ has four unimodular-triangle facets and four square facets that are unimodularly equivalent to
a unit square, and so
$
  \ehr_{ \P_m } (j) = \left( \frac m 6 + 1 \right) j^3 + 3 \, j^2 + c \, j + 1
$ (see, e.g., \cite[Theorem 5.6]{Discretely}).
Since $\ehr_{ \P_m } (1) = 8$, we compute $c = 3 - \frac m 6$. The lemma follows.
\end{proof}

\begin{lemma}\label{dimensionstop}
Let $\P$ be a (not necessarily very ample) lattice polytope of dimension $d$ and $k_0\geq d-1$ be an integer. If $(M_\P)_{k_0}=(\overline{M_\P})_{k_0}$, then $(M_\P)_k=(\overline{M_\P})_k$ for all $k\geq k_0$.
\end{lemma}
\begin{proof}
The lemma follows from the fact that $\overline{M_\P}$ is generated as an \Def{$M_\P$-module} by elements of degree at most $d-1$, i.e.,
\[
\overline{M_\P} \ =\bigcup_{
\begin{matrix}
\m\in\overline{M_\P}\\
\deg\m\le d-1
\end{matrix}
}(\m+M) \, ;
\]
see \cite[Corollary 1.3.4]{BrGuTr} and its proof.
\end{proof}

\begin{proof}[Proof of Theorem \ref{gapvectorofPm}]
First we prove the following formula for the Hilbert function:
\[
\#\left(M_{\P_m}\right)_j=(j+1){j+3\choose 3},\qquad 1\leq j\leq m-1.
\]
We fix $1\leq j\leq m-1$. Consider a point $S=\sum_{i=1}^j T_i=(e,f,g)\in\ZZ^3$ for some lattice points $T_i\in P_m$. We use the following notation:
$$A_0=(0,0,0),\quad A_1=(0,0,1),\quad B_0=(1,0,0),\quad B_1=(1,0,1)$$
$$C_0=(0,1,0),\quad C_1=(0,1,1),\quad D_0=(1,1,m),\quad D_1=(1,1,m+1).$$
Let $d$ be the number of times the points $D_i$ occur in some decomposition of $S$. We have $dm\leq g\leq dm+j$, so $d=\lfloor \frac{g}{m}\rfloor$. Thus the number
of times the points $B_i$ and $C_i$ occur in the decomposition equals respectively $b:=e- \lfloor \frac{g}{m}\rfloor$ and $c:=f-\lfloor \frac{g}{m}\rfloor$. Hence,
the points $A_i$ must occur $a:=j-b-c-d$ times. In particular, each decomposition of $S$ has the same number of occurrences of the points in each group $A_i$, $B_i$,
$C_i$ and $D_i$. Moreover, the points $A_1,B_1,C_1,D_1$ must occur $(g\mod m)\leq j$ times. The numbers $a,b,c,d$ and $(g\mod m)$ uniquely determine the point $S$.
We obtain a bijection between the points of $(M_{\P_m})_j$ and points of the form $(a,b,c,d,h)\in\ZZ^5$ where $a,b,c,d,h\geq 0, \ a+b+c+d=j, \ h\leq j$.

Thus $\#\left(M_{\P_m}\right)_j$ equals the number of ordered partitions of $j$ into four parts, for the choice of $a,b,c,d$, times $j+1$ for the choice of $(g\mod m)$. This equals $(j+1){j+3\choose 3}$.
So by Lemma \ref{highgaps}, for $j<m$ we obtain
\begin{align*}
\gv_j(\P_m)&=\ehr_{P_m}(j)-\#\left(M_{\P_m}\right)_j\\
&=\left( \tfrac m 6 + 1 \right) j^3 + 3 \, j^2 + \left( 3 - \tfrac m 6 \right)  j + 1-(j+1){j+3\choose 3}\\
&={j+1\choose 3}(m-j-1) \, .
\end{align*}
Now the formula for $\gv_k(\P_m)$ follows by Lemma \ref{dimensionstop} because $\ehr_{P_m}(m-1)-\#\left(M_{\P_m}\right)_{m-1}=0$ and $m\geq 3=\dim(\P_m)$. The rest of Theorem \ref{gapvectorofPm} follows easily from this formula.
\end{proof}

Based on Theorem \ref{gapvectorofPm} and several other gap vectors of very ample polytopes we computed by effective methods, we offer the following conjecture.
\begin{conjecture}\label{gapvectorconj}
The gap vector of any very ample polytope $\P$ that has normal facets is unimodal, i.e., there exists $j$ such that
\[
  \gv_1(\P) \le \gv_2(\P) \le \dots \le \gv_j(\P) \ge \gv_{ j+1 }(\P) \ge \dots \ge \gv_{ \gamma(\P) }(\P) \, .
\]
\end{conjecture}

%Conjecture \ref{gapvectorconj}(a) for very ample 3-polytopes follows from Lemma \ref{dimensionstop}.

The reason we require that the facets of $\P$ are normal in Conjecture \ref{gapvectorconj} (which in this case is equivalent to the facets of $\P$ being
integrally closed and automatically satisfied for very ample $3$-polytopes) is that if the gap vectors of the facets of $\P$ contribute to $\gv(\P)$ in a nontrivial
way, then---because the former are relatively independent of each other---the resulting interference
can cause oscillation in $\gv(\P)$. Very recently explicit examples of this phenomenon appeared in
\cite{JaMichalVeryAmple}: there exist very ample polytopes with gap vectors having only two nonzero entries at two arbitrary indices.

%*************************************************

\section{Koszul segmental fibrations}\label{Koszul}

For a field $\kk$, a narrower class of Koszul $\kk$-algebras than those admitting quadratic Gr\"obner bases is formed by the homogeneous $\kk$-algebras for which the defining ideal $I$ admits a square-free quadratic Gr\"obner basis. In the special case of algebras of the form $\kk[\cA]$, where $\cA\subset\ZZ^d$ is a point
configuration, the existence of such Gr\"obner bases is a purely combinatorial condition due to Sturmfels (see \cite{STURMpol} or \cite[Sections 7.A,B]{Kripo})---Theorem \ref{Sturmfels} below.

To describe the connection with polytopal combinatorics, we recall the relevant terminology. We refer the reader
to \cite[Sections 1.D,E,F]{Kripo} for background on polytopal complexes, regular subdivisions and triangulations, stars and links in simplicial
complexes, etc.

A polytopal subdivision $\Delta$ of a polytope $\P\subset\RR^d$ is \Def{regular} if there is a convex function
$h:\P\to\RR$ (i.e., $h(\lambda \x+\mu \y)\le \lambda h(\x)+\mu h(\y)$ for all $\x,\y\in \P$ and
$\lambda,\mu\in\RR_{\ge0}$ with $\lambda+\mu=1$) whose domains of linearity are exactly the facets of $\Delta$,
i.e., the maximal faces of $\Delta$ are the maximal subsets of $\P$ on which $h$ restricts to an affine map. We
say that $h$ is a \Def{support function} for~$\Delta$.

A simplicial complex $\Delta$ is \Def{flag} if the minimal non-faces of $\Delta$ are pairs of vertices of $\Delta$.
(A \Def{non-face} of the simplicial complex $\Delta$ with vertex set $V$ is a subset $W\subset V$ with $W\notin\Delta$.)
In general, the \Def{degree} of a simplicial complex $\Delta$ on a vertex set $V$ is
\[
  \deg(\Delta):=\max\big(\#W\ |\ W\ \text{a minimal non-face of}\ \Delta\big) \, ;
\]
see \cite{BrGuTr}. Thus, flag simplicial complexes are exactly the simplicial complexes of degree~2.

A triangulation of a polytope is thought of as the corresponding geometric simplicial complex, as opposed to the underlying abstract simplicial complex, i.e., the elements of the triangulation are simplices in the ambient Euclidean space.

A triangulation $\Delta$ of a lattice polytope $\P\subset\RR^d$ is \Def{unimodular} if the simplices in $\Delta$ are all unimodular.

Observe that, in exploring ring-theoretical properties of $\kk[\cA]$, there is no loss of generality in assuming that
$\L(\cA)=\ZZ^d$: the isomorphism class of $\kk[\cA]$ is independent of the ambient lattice for $\cA$ and,
therefore, it can be chosen to be~$\L(\cA)$.

\begin{theorem}[Sturmfels \cite{STURMpol}]\label{Sturmfels}
For a point configuration $\cA = \{\a_1,\ldots,\a_N\} \subset\ZZ^d$ with $\L(\cA)=\ZZ^d$, the binomial ideal
\begin{align*}
I_{\cA}:=\Ker\left(
  \begin{array}{rcl}
    \kk[X_1,\ldots,X_N] & \to & \kk[\cA] \\
    X_i & \mapsto & \a_i
  \end{array}
  \right)\subset\kk[&X_1,\ldots,X_N]
\end{align*}
admits a square-free quadratic Gr\"obner basis if and only if there is a regular unimodular flag triangulation of
$\conv(\cA)$ with the vertex set $\cA$.
\end{theorem}

%In particular, for a point configuration $\cA\subset\ZZ^d$ satisfying $\L(\cA)=\ZZ^d$, the existence of a square-free quadratic Gr\"obner basis for $I_{\cA}$ implies the equality $\cA=\conv(\cA)\cap\ZZ^d$ and that the polytope $\conv(A)$ is integrally closed.

By \cite[Ch. III]{Toroidal}, for every lattice polytope $\P$, the dilated polytope $c \, \P$ has a regular unimodular
triangulation for some $c\in\NN$. By \cite[Theorem 1.4.1]{BrGuTr}, for a
lattice polytope $\P\subset\RR^d$, the ideal $I_{c \P\cap\ZZ^d}$ has a quadratic (but possibly not square-free)
Gr\"obner basis whenever $c\ge\dim \P$. Thus, informally speaking, there is no algebraic obstruction to the
existence of regular unimodular flag triangulations of the dilated lattice polytopes $c \, \P$ for $c\ge\dim \P$. However, currently even the existence of dimensionally uniform lower
bounds for the factors $c$ such that the polytopes $c \, \P$ have unimodular triangulations is a major open problem \cite{Santos}.

Theorem \ref{koszulclass} below leads to a large class of polytopes admitting triangulations with all the nice properties. As we explain later on, this theorem could have been included in
\cite{DHZ}---the argument in \cite[Section 4.2]{DHZ}, used there for a rather
special case of \emph{Nakajima polytopes}, works also in the general case. A
related discussion can be found in Haase--Paffenholz's report in \cite{mfo}.

\begin{theorem}\label{koszulclass}
Let $f:\P\to \Q$ be a lattice segmental fibration of lattice polytopes. Assume $\Delta$ is a regular unimodular flag triangulation of $\Q$ such that the image $f(F)$ of every face $F\subset \P$ is a union of faces of $\Delta$. Then $\P$ has a regular unimodular flag triangulation; in particular, $\P$ is integrally closed and Koszul.
\end{theorem}

Observe that the polytope $\P_m$ in Theorem \ref{highgaps} satisfies the additional condition in Theorem
\ref{koszulclass} (with respect to both triangulations of the unit square as the polytope $\Q$ simultaneously) if and only if $m=0$.

Before outlining the proof of Theorem \ref{koszulclass} we discuss some explicit classes of polytopes this theorem leads to.

\begin{example}[\emph{Nakajima polytopes}]\label{Example}
Assume $\Q\subset\RR^d$ is a lattice polytope and $\alpha,\beta:\Q\to\RR$ are affine maps such that $\alpha(\x),\beta(\x)\in\ZZ$ for all $x\in\Q\cap\ZZ^d$ and $\alpha\le\beta$ on $\Q$. Consider the lattice polytope
$$
\Q(\alpha,\beta):=\conv\big((\x,y)\ |\ \x\in \Q,\ \alpha(\x)\le y\le\beta(\x)\big)\subset\RR^{d+1}.
$$
Then the orthogonal projection $f:\Q(\alpha,\beta)\to \Q$ and any regular unimodular flag triangulation of $\Q$ satisfy the conditions in Theorem \ref{koszulclass}. It is easily seen that $\Q(\alpha,\beta)\cong\Q(0,\beta-\alpha)$ as lattice polytopes.

Iteratively using the $\Q(\alpha,\beta)$-construction, starting with a point, we get exactly the class of polytopes characterized in \cite{Nakajima} as the polytopes $\P$ for which the (complex) \emph{affine} toric variety $\Spec(\CC[M_\P])$ is a local compete intersection. In \cite{DHZ} these polytopes are called \emph{Nakajima polytopes}.
It is clear that the polytopes $\P$ in Theorem \ref{koszulclass} can have arbitrarily more complicated shapes than the ones resulting from the $\Q(\alpha,\beta)$-construction.

The smooth polytopes of type $\P(I_1,I_2,I_3,I_4)$ in Section \ref{Class} are of type $\Q(\alpha,\beta)$ and, therefore, integrally closed and Koszul. More generally, if $Q$ is any smooth polytope and $\alpha,\beta:\Q\to\RR$ are as above, satisfying the stronger condition $\alpha<\beta$ on $\Q$, then $Q(\alpha,\beta)$ is smooth as well.

In particular, iteratively using the $\Q(\alpha,\beta)$-construction with $\alpha<\beta$ on $\Q$, starting with a point, we get smooth Nakajima polytopes in arbitrary dimensions, all combinatorially equivalent to cubes but representing infinitely many affine equivalence classes.

Starting the iteration with other smooth polytopes that admit triangulations with the desired properties, we get richer classes of higher-dimensional smooth Koszul polytopes. For instance, by \cite[Corollary 3.2.5]{BrGuTr}, all lattice smooth polygons can serve as the initial input of this machine.
\end{example}

\begin{example}[\emph{Lattice $A$-fibrations}]\label{Example2} Let $\P\subset\RR^d$ be a lattice polytope cut out by a root system of type $A$. In other words, $\P$
is bounded by hyperplanes parallel to hyperplanes of the form $X_i=0$ and $X_i=X_j$, $1\le i\not=j\le d$ (in the language of \cite{lampostnikov}, $\P$ is an 
\emph{alcoved polytope}). Then $\P$ has a \emph{canonical} nice triangulation $\Delta(\P)$ satisfying the compatibility condition: if $\dim \P>0$ then its orthogonal projection  $\Q\subset\RR^{d-1}$ is also cut out by a root system of type $A$ and the projection $f:\P\to \Q$ satisfies the condition in Theorem~\ref{koszulclass} with respect to $\Delta(\Q)$.

In fact, it was shown in \cite[Section 2]{BrGuTr} that each polytope from the above-mentioned class is nicely triangulated by cutting it along the integer translates
of the coordinate hyperplanes and the hyperplanes of the form $X_i-X_j$. So it is enough to show the following:

\medskip\noindent\emph{Claim.} Let $\P\subset\RR^d$ be a lattice polytope cut out by a root system of type $A$. Then its
orthogonal projection $\Q$ in $\RR^{d-1}$ is also a lattice polytope cut out by a root system of type $A$. 

\medskip Any facet of $\Q$ is the orthogonal projection of either a facet of $\P$ or a codimension-2 face of
$\P$. In the first case the corresponding support hyperplane of $\P$ is of the form $X_i=a$ or $X_i-X_j=b$ for some $a,b\in\ZZ$ and $1\le i\not=j\le d-1$. In particular,
the same equality defines the image facet of $\Q$. In the second case the codimension-2 face of $\P$ in question corresponds to a system of type either $X_d=a$ and
$X_d-X_i=b$ or $X_d-X_i=a$ and $X_d-X_j=b$, where $a,b\in\ZZ$ and $1\le i\not= j\le d-1$. Correspondingly, the image facet of $\Q$ is defined by either $X_i=a-b$ or $X_i-X_j=b-a$. \qed

\medskip One can extend the notion of lattice segmental fibrations as follows: assume $\P\subset\RR^{d_1}$ and $\Q\subset\RR^{d_2}$ are lattice polytopes and $f:\P\to\Q$ is an affine map; call $f$ a \Def{lattice $A$-fibration} if it satisfies the conditions
\begin{enumerate}[{\rm (i)}]
\item $f^{-1}(\x)$ is a lattice polytope for every $\x\in\Q\cap\ZZ^{d_2}$,
\item $\P\cap\ZZ^{d_1}\subset\bigcup_{\Q\cap\ZZ^{d_2}}f^{-1}(\x)$,
\item there is a full-rank affine map $\pi:\P\to\RR^{\dim\P-\dim\Q}$, injective on $f^{-1}(\x)$ for every $\x\in\Q$ and such that $\pi$ induces a surjective group homomorphism onto
$\ZZ^{\dim\P-\dim\Q}$ and $\pi\big(f^{-1}(\x)\big)$ is a lattice polytope, cut out by a root system of type $A$, for every $\x\in\Q\cap\ZZ^{d_2}$.
\end{enumerate}

The class of $A$-fibrations is considerably larger than that of segmental fibrations and, in general, $\P$ is very different from a Nakajima polytope even for simple $\Q$ (e.g., a segment). Let $f:\P\to\Q$ be a lattice $A$-fibration, where $\P\subset\RR^{d_1}$ and $\Q\subset\RR^{d_2}$. In view of the claim above, applied iteratively to the fibers over the lattice points of $\Q$, the map $f$ factors through segmental fibrations
\begin{equation}\label{sequenceoffibrations}
\P\mathrel{\mathop{\longrightarrow}^{\phi_0}}\P_1\mathrel{\mathop{\longrightarrow}^{\phi_1}}\cdots\mathrel{\mathop{\longrightarrow}^{\phi_{k-1}}}
\P_k\mathrel{\mathop{\longrightarrow}^{\phi_k}}\Q \qquad \text{ where } \qquad k=\max\big(\dim f^{-1}(\x)\ |\ \x\in\Q\cap\ZZ^{d_2}\big).
\end{equation}
So one can ask whether Theorem \ref{koszulclass} can be extended to lattice $A$-fibrations. The
obstruction to iteration of Theorem \ref{koszulclass} is that it is not clear whether the condition
on faces in that theorem can be kept under control at each step from $\phi_k$ to $\phi_0$. In fact,
the triangulation of $\P$ resulting from the proof of Theorem \ref{koszulclass}, when both $\P$ and $\Q$ are cut out by a root system of type $A$, is \emph{not} the same as $\Delta(\P)$, not even if $\dim\P=\dim\Q+1$.

However, there is a big subclass of lattice $A$-fibrations for which Theorem \ref{koszulclass} can be iterated along the corresponding sequences (\ref{sequenceoffibrations}). Call a lattice $A$-fibration a \emph{lattice cubical fibration} if in the condition (iii) above we require that the polytope $f^{-1}(\x)$ has the facets parallel to coordinate hyperplanes.
One can easily check that the proof of Theorem \ref{koszulclass} allows one to control the condition on the faces in the theorem at each step from $\phi_k$ to $\phi_0$ when $f:\P\to\Q$ is cubical. Therefore, Theorem \ref{koszulclass} extends to lattice cubical fibrations.
\end{example}

\begin{proof}[Sketch of the proof of Theorem \ref{koszulclass}]
Following the approach in \cite[Section 4.2]{DHZ}, we first take the regular polytopal subdivision $R:=(f^{-1}(\delta))_{\delta\in\Delta}$ of $\P$ and then refine it to a triangulation, using successive stellar subdivisions by the lattice points in $\P$ in any linear order. The regularity, flag, and unimodularity properties of the final outcome are checked exactly the same way as in \cite{DHZ}.

Our original approach (the one we used before we learned about the overlap with \cite{DHZ}) produces different triangulations of $\P$, also refining the polytopal subdivision $R$ but without involving stellar subdivisions. Below we describe the construction.

For a closed subset $Y\subset\RR^{d+1}$, we put
\begin{align*}
&Y^+:=\{\y\in Y\ |\ \y\ \text{has the largest}\ (d+1)\text{st coordinate within}\ f^{-1}(f(\y))\},\\
&Y^-:=\{\y\in Y\ |\ \y\ \text{has the smallest}\ (d+1)\text{st coordinate within}\ f^{-1}(f(\y))\}.
\end{align*}

There is no loss of generality in assuming that $\Q\subset\RR^d$, $\dim \Q=d$, $\P\subset\RR^{d+1}$, and
$f$ is the projection onto the first $(d+1)$-coordinates. % We also put $\e_{d+1}=(0,\ldots,0,1)\in\RR^{d+1}$.
We can assume $(\P\setminus \P^-)\cap\ZZ^{d+1}=\{\y_1,\ldots,\y_r\}$, where
$$
f(\y_i)=f(\y_j)\quad \text{ and } \quad  (\y_i)_{d+1}< (\y_j)_{d+1}\ \quad \text{ imply } \quad i < j \, .
$$

Define the sequence of polytopal complexes $\Pi_0,\Pi_1,\ldots,\Pi_r$ inductively as follows:
\begin{enumerate}[{\rm$\bullet$}]
\item $\Pi_0=\{f^{-1}(\delta)\cap \P^-\}_\Delta$
\item $\Pi_k=\big\{\conv(\y_k,F)\ |\ F\in\starr^+_{\Pi_{k-1}}(\y_k-\e_{d+1})\big\}\cup\Pi_{k-1}$, where
\begin{align*}
\starr^+_{\Pi_{k-1}}(\y_k-\e_{d+1})=\{\tau\in\starr_{\Pi_{k-1}}(\y_k-\e_{d+1})\ |\ \tau&\subset|\Pi_{k-1}|^+\},\\
&k=1,\ldots,r.
\end{align*}
\end{enumerate}
\noindent($|\dots|$ denotes the support of the polytopal complex in question.) That $\Pi_r$ is a triangulation of $\P$ with the desired properties can be shown along the same lines as for the triangulations in \cite{DHZ} (only the regularity needs a minor change in the argument).
\end{proof}

It is interesting to notice that the triangulations $\Pi_r$ are usually \emph{different} from those in \cite[Section 4.2]{DHZ} when the fibers $f^{-1}(\x)$ contain at least $4$ lattice points for several $\x\in\Q\cap\ZZ^d$.

\clearpage

\begin{figure}[h!]
\flushleft
\caption{Two triangulations $\Pi_r$ of $\P$ for two different enumerations of $(\P\setminus\P^-)\cap\ZZ^{d+1}$.}
\includegraphics[trim = 0mm 0in 0mm 1.5in, clip, scale=.4]{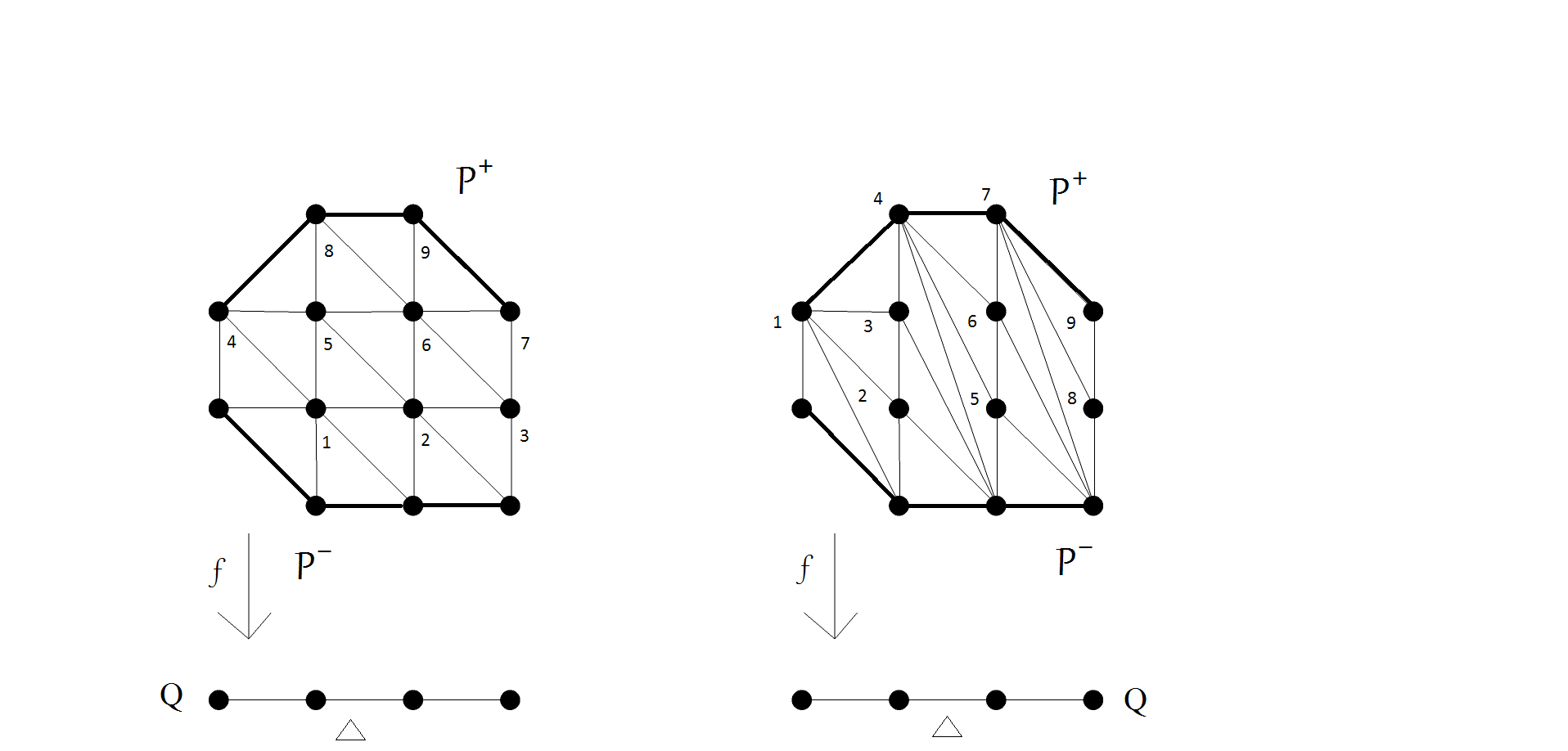}
\end{figure}

\bibliography{bibliography}

\begin{thebibliography}{10}

\bibitem{aim}
{W}orkshop: {C}ombinatorial {C}hallenges in {T}oric {V}arieties.
\newblock April 27 to May 1, 2009, organized by Joseph Gubeladze, Christian
  Haase, and Diane Maclagan, American Institute of Mathematics, Palo Alto.
  \textsf{http://www.aimath.org/pastworkshops/toricvarieties.html}.

\bibitem{mfo}
Mini-{W}orkshop: {P}rojective {N}ormality of {S}mooth {T}oric {V}arieties.
\newblock {\em Oberwolfach Rep.}, 4(3):2283--2319, 2007.
\newblock Abstracts from the mini-workshop held August 12--18, 2007, organized
  by Christian Haase, Takayuki Hibi and Diane Maclagan, Oberwolfach Reports.
  Vol. 4, no. 3.

\bibitem{Discretely}
Matthias Beck and Sinai Robins.
\newblock {\em Computing the continuous discretely}.
\newblock Undergraduate Texts in Mathematics. Springer, 2007.

\bibitem{8authors}
Tristram Bogart, Christian Haase, Milena Hering, Benjamin Lorenz, Benjamin
  Nill, Andreas Paffenholz, Francisco Santos, and Hal Schenck.
\newblock Few smooth $d$-polytopes with ${N}$ lattice points.
\newblock {\em Israel J. Math.}, to appear.

\bibitem{BRUNSquest}
Winfried Bruns.
\newblock The quest for counterexamples in toric geometry.
\newblock In {\em {P}roc. {CAAG} 2010}, volume~17, pages 1--17.

\bibitem{uncovered}
Winfried Bruns and Joseph Gubeladze.
\newblock Normality and covering properties of affine semigroups.
\newblock {\em J. Reine Angew. Math.}, 510:161--178, 1999.

\bibitem{Kripo}
Winfried Bruns and Joseph Gubeladze.
\newblock {\em Polytopes, rings, and {$K$}-theory}.
\newblock Springer Monographs in Mathematics. Springer, 2009.

\bibitem{BrGuTr}
Winfried Bruns, Joseph Gubeladze, and Ng{\^o}~Vi{\^e}t Trung.
\newblock Normal polytopes, triangulations, and {K}oszul algebras.
\newblock {\em J. Reine Angew. Math.}, 485:123--160, 1997.

\bibitem{Vetter}
Winfried Bruns, J{\"u}rgen Herzog, and Udo Vetter.
\newblock Syzygies and walks.
\newblock In {\em Commutative algebra ({T}rieste, 1992)}, pages 36--57. World
  Sci. Publ., River Edge, NJ, 1994.

\bibitem{Caviglia}
Giulio Caviglia.
\newblock The pinched {V}eronese is {K}oszul.
\newblock {\em J. Algebraic Combin.}, 30:539--548, 2009.

\bibitem{CONCA}
Aldo Conca, Emanuela De~Negri, and Maria~Evelina Rossi.
\newblock Koszul algebras and regularity.
\newblock In {\em Commutative algebra}, pages 285--315. Springer, New York,
  2013.

\bibitem{TORICvarieties}
David~A. Cox, John~B. Little, and Henry~K. Schenck.
\newblock {\em Toric varieties}, volume 124 of {\em Graduate Studies in
  Mathematics}.
\newblock American Mathematical Society, Providence, RI, 2011.

\bibitem{DHZ}
Dimitrios~I. Dais, Christian Haase, and G{\"u}nter~M. Ziegler.
\newblock All toric local complete intersection singularities admit projective
  crepant resolutions.
\newblock {\em Tohoku Math. J. (2)}, 53:95--107, 2001.

\bibitem{Froberg}
Ralf Fr{\"o}berg.
\newblock Koszul algebras.
\newblock In {\em Advances in commutative ring theory ({F}ez, 1997)}, volume
  205 of {\em Lecture Notes in Pure and Appl. Math.}, pages 337--350. Dekker,
  New York, 1999.

\bibitem{Algeo}
Robin Hartshorne.
\newblock {\em Algebraic geometry}.
\newblock Springer-Verlag, 1977.
\newblock Graduate Texts in Mathematics, No. 52.

\bibitem{Hering}
Milena Hering.
\newblock Multigraded regularity and the {K}oszul property.
\newblock {\em J. Algebra}, 323:1012--1017, 2010.

\bibitem{Akihiro}
Akihiro Higashitani.
\newblock Non-normal very ample polytopes and their holes.
\newblock {\em Electron. J. Combin.}, 21:Paper 1.53, 12, 2014.

\bibitem{katthan}
Lukas Katth{\"a}n.
\newblock Polytopal affine semigroups with holes deep inside.
\newblock {\em Discrete Comput. Geom.}, 50:503--508, 2013.

\bibitem{Toroidal}
George Kempf, Finn~Faye Knudsen, David Mumford, and Bernard Saint-Donat.
\newblock {\em Toroidal embeddings. {I}}.
\newblock Lecture Notes in Mathematics, Vol. 339. Springer-Verlag, 1973.

\bibitem{lampostnikov}
Thomas Lam and Alexander Postnikov.
\newblock Alcoved polytopes. {I}.
\newblock {\em Discrete Comput. Geom.}, 38(3):453--478, 2007.

\bibitem{JaMichalVeryAmple}
Micha{\l} Laso\'n and Mateusz Micha{\l}ek.
\newblock Non-normal, very ample polytopes---constructions and examples.
\newblock Preprint: http://arxiv.org/abs/1406.4070.

\bibitem{Nakajima}
Haruhisa Nakajima.
\newblock Affine torus embeddings which are complete intersections.
\newblock {\em Tohoku Math. J. (2)}, 38:85--98, 1986.

\bibitem{Toric}
Tadao Oda.
\newblock {\em Convex bodies and algebraic geometry}, volume~15 of {\em
  Ergebnisse der Mathematik und ihrer Grenzgebiete (3) [Results in Mathematics
  and Related Areas (3)]}.
\newblock Springer, 1988.

\bibitem{Ohsugi}
Hidefumi Ohsugi, J{\"u}rgen Herzog, and Takayuki Hibi.
\newblock Combinatorial pure subrings.
\newblock {\em Osaka J. Math.}, 37:745--757, 2000.

\bibitem{PAYNEroot}
Sam Payne.
\newblock Lattice polytopes cut out by root systems and the {K}oszul property.
\newblock {\em Adv. Math.}, 220:926--935, 2009.

\bibitem{Peeva}
Irena Peeva.
\newblock Infinite free resolutions over toric rings.
\newblock In {\em Syzygies and {H}ilbert functions}, volume 254 of {\em Lect.
  Notes Pure Appl. Math.}, pages 233--247. 2007.

\bibitem{Priddy}
Stewart~B. Priddy.
\newblock Koszul resolutions.
\newblock {\em Trans. Amer. Math. Soc.}, 152:39--60, 1970.

\bibitem{DIOPHANTINE}
Jorge~L. Ram{\'{\i}}rez~Alfons{\'{\i}}n.
\newblock {\em The {D}iophantine {F}robenius problem}, volume~30 of {\em Oxford
  Lecture Series in Mathematics and its Applications}.
\newblock Oxford University Press, 2005.

\bibitem{Santos}
F.~Santos and G.~M. Ziegler.
\newblock Unimodular triangulations of dilated 3-polytopes.
\newblock {\em Trans. Moscow Math. Soc.}, pages 293--311, 2013.

\bibitem{STURMpol}
Bernd Sturmfels.
\newblock {\em Gr\"obner bases and convex polytopes}, volume~8 of {\em
  University Lecture Series}.
\newblock American Mathematical Society, 1996.

\end{thebibliography}
\bibliographystyle{plain}

\end{document}